\documentclass[10pt]{amsart}
\usepackage{amsmath,amssymb}
\usepackage[utf8]{inputenc}
\usepackage{mathtools}
\usepackage{mathrsfs}
\usepackage{float}
\usepackage{multicol}
\usepackage{caption}
\usepackage{todonotes}
\usepackage{verbatim}

\newcommand{\rdecomp}{\xrightarrow{R}}

\usepackage{amsthm}
\newtheorem{theorem}{Theorem}[section]
\newtheorem*{thm*}{Theorem}
\newtheorem*{decomposition}{Decomposition Theorem}

\newtheorem*{fibering}{Fibering Permanence}
\newtheorem*{limit}{Limit Permanence}
\newtheorem*{product}{Product Permanence}

\newtheorem{corollary}[theorem]{Corollary}
\newtheorem{lemma}[theorem]{Lemma}

\newtheorem{question}[theorem]{Question}
\theoremstyle{definition}
\newtheorem{definition}[theorem]{Definition}
\newtheorem*{def*}{Definition}

\newcommand{\U}{\mathcal{U}}
\newcommand{\V}{\mathcal{V}}

\newcommand{\df}[1]{{{\bf #1}}}
\newcommand{\C}[1]{\mathfrak{C}_{#1}}
\newcommand{\APC}{\mathfrak{Apc}}
\DeclareMathOperator{\supp}{supp}
\DeclareMathOperator{\diam}{diam}

\renewcommand{\epsilon}{\varepsilon}

\begin{document}
\bibliographystyle{abbrv}

\title{Decomposition theorems for asymptotic property~C and property~A}
\author{G.~Bell}
\address[G.~Bell]{Department of Mathematics and Statistics, UNC Greensboro, Greensboro, NC 27412, USA}
\email[G.~Bell]{gcbell@uncg.edu}
\thanks{This reserach was supported by a UNCG Faculty First Grant. }
\author{D.~G\l{}odkowski}
\address[D.~G\l{}odkowski and A.~Nag\'orko]{Department of Mathematics, Informatics, and Mechanics, University of Warsaw, Banacha 2, 02-097 Warsaw, Poland}
\email[D.~G\l{}odkowski]{damiang0071@gmail.com}
\author{A.~Nag\'orko}
\email[A.~Nag\'orko]{amn@mimuw.edu.pl}
\thanks{This research was supported by the NCN (Narodowe Centrum Nauki) grant no. 2011/01/D/ST1/04144.}
\date{February 12, 2019}






\begin{abstract}
We combine aspects of the notions of finite decomposition complexity and asymptotic property C into a notion that we call finite APC-decomposition complexity. Any space with finite decomposition complexity has finite APC-decomposition complexity and any space with asymptotic property C has finite APC-decomposition complexity. Moreover, finite APC-decomposition complexity implies property A for metric spaces. We also show that finite APC-decomposition complexity is preserved by direct products of groups and spaces, amalgamated products of groups, and group extensions, among other constructions.
\end{abstract}
\maketitle

\section{Introduction}

Dranishnikov introduced asymptotic property C (APC) for metric spaces as a large-scale analog of topological property C~\cite{Dr00}. APC is a weaker condition than finite asymptotic dimension~\cite{Gr93}, but it is strong enough that discrete metric spaces with asymptotic property~C have Yu's property~A~\cite{Yu2}.

Later, Guentner, Tessera, and Yu introduced another large-scale property of metric spaces called finite decomposition complexity (FDC)~\cite{Guentner-Tessera-Yu-2012,Guentner-Tessera-Yu-2013}. FDC is again a weaker condition than finite asymptotic dimension, but is still sufficiently strong to imply Yu's property A. 

Both APC and FDC have received a lot of attention recently~\cite{beckhardt2017,beckhardt-goldfarb2016,bell-moran2015,bell-moran-nagorko2016,bn-stability,DrZ14,DrZ17,dydak2016,regular-fdc,yamauchi2015}. The current paper is a natural continuation of the decomposition lemma that enabled the first and third authors to prove that asymptotic property C is preserved by free products~\cite{bn-stability}. Similar ideas of applying FDC-like decompositions to coarse properties (including APC) recently appeared in a paper of Dydak~\cite{dydak2016}. Some of our corollaries can also be deduced as special cases of Dydak's theorems; however, unlike Dydak's approach, we introduce a notion of decomposition depth and provide upper bound estimates on this depth in our permanence results.

We combine the decomposition ideas from FDC with Dranishnikov's APC into a concept we call finite APC-decomposition complexity. More precisely, we study the permanence of asymptotic property C and of property A 
with respect to the decomposition notion given in Definition~\ref{def:rdecomp}.

\begin{definition}
Let $(X,d)$ be a metric space and let $r$ be a real number. We say that the family $\mathcal{U}$ of metric subspaces of $X$ is \df{$r$-disjoint} if $d(x,y)>r$ whenever $x\in U$, $y\in U'$ and $U\neq U'$ are elements of $\U$. 
\end{definition}

\begin{definition}\label{def:rdecomp}
Let $\mathcal{X}$ and $\mathcal{Y}$ be families of metric spaces.
Let $R \in \mathbb{R}^\mathbb{N}$.
We say that $\mathcal{X}$ is \df{uniformly $R$-decomposable} over $\mathcal{Y}$ if there exists an integer~$k$ such that
for each $X \in \mathcal{X}$ there exists a sequence $\U_1, \U_2, \ldots, \U_{k}$ of subsets of $\mathcal{Y}$ such that each $\U_i$ is $R_i$-disjoint
and $\bigcup_i \U_i$ covers~$X$.
We denote this by $\mathcal{X} \rdecomp \mathcal{Y}$.
\end{definition}

\begin{definition}
A family of metric spaces $\mathcal{B}$ will be said to be \df{bounded} if there is some positive real number $D$ so that $\sup\{\diam(B)\colon B\in\mathcal{B}\}<D$. This is sometimes described as being \df{uniformly bounded}. 
\end{definition}

Using this notion and Definition~\ref{def:rdecomp}, we can rephrase the definition of asymptotic property C in the following way.

\begin{definition}
We say that a family $\mathcal{X}$ of metric spaces has \df{uniform asymptotic property~C} if for each $R \in \mathbb{R}^\mathbb{N}$ there exists a bounded family $\mathcal{Y}_R$ of metric spaces such that $\mathcal{X} \rdecomp \mathcal{Y}_R$.
We say that a metric space $X$ has \df{asymptotic property~C} (APC) if the family $\{ X \}$ has uniform asymptotic property~C.
\end{definition}

The following is a decomposition theorem for uniform asymptotic property~C.

\begin{theorem}\label{thm:decomposition C}
  Let $\mathcal{Y}$ be a family of metric spaces with uniform asymptotic property C.
  If a family $\mathcal{X}$ of metric spaces admits a uniform $R$-decomposition over $\mathcal{Y}$ for each sequence $R$, then $\mathcal{X}$ has uniform asymptotic property C.
\end{theorem}

Note the order of the quantifiers in the assumption:
$\exists \mathcal{Y} \forall R \mathcal{X} \rdecomp \mathcal{Y}$.
The family $\mathcal{Y}$ does not depend on the sequence $R$.
A special case of Theorem~\ref{thm:decomposition C} was proven for $\mathcal{Y}$ equal to the family of $k$-dimensional subsets of a single space $X$, under the name Decomposition Lemma~\cite{bn-stability}. As mentioned above, this decomposition lemma was used to prove that the free product $X \ast Y$ of two discrete metric spaces $X$ and $Y$ with APC has APC.

We define uniform property~A for a family of spaces in the following way.
\begin{definition}
We say that a map $\xi \colon X \to \ell^1(X)$ has \df{$\varepsilon$-variation} if for each $k \in \mathbb{N}$ and each $x_1,x_2 \in X$ with $d(x_1, x_2) \leq k$ we have $\| \xi_{x_1} - \xi_{x_2} \|_1 \leq k\varepsilon$.
\end{definition}

\begin{definition}
Let $X$ be a discrete metric space. We say that $X$ has \df{bounded geometry} if for every $r>0$ there is a number $N_r$ such that $|B(x,r)|<N_r$ for all $x\in X$.
\end{definition}

\begin{definition}
Let $\mathcal{X}$ be a family of discrete metric spaces with bounded geometry.
We say that $\mathcal{X}$ has \df{uniform property A} if for each $\varepsilon > 0$ there exists $S > 0$ such that for each $X \in \mathcal{X}$ there exists a map $\xi \colon X \to \ell^1(X)$ such that 
\begin{enumerate}
\item $\| \xi_x \|_1 = 1$ for all $x \in X$,
\item $\xi$ has $\varepsilon$-variation,
\item $\supp \xi_x \subset \bar B(x, S)$ for all $x \in X$.
\end{enumerate}
\end{definition}

The following is a decomposition theorem for uniform property~A.

\begin{theorem}\label{thm:decomposition A}
  For each sequence $R$ let $\mathcal{Y}_R$ be a family of metric spaces with uniform property A.
  If a family $\mathcal{X}$ of metric spaces admits a uniform $R$-decomposition over $\mathcal{Y}_R$ for each sequence $R$, then $\mathcal{X}$ has uniform property A.
\end{theorem}

Note the order of quantifiers in the assumption:
$\forall R \exists \mathcal{Y}_R \colon \mathcal{X} \rdecomp \mathcal{Y}_R$
The family $\mathcal{Y}_R$ may vary with a change of the sequence $R$.

Following methods used to describe finite decomposition complexity~\cite{Guentner-Tessera-Yu-2013} we use the notion of uniform $R$-decomposition to define complexity classes $\C{\alpha}$ for each ordinal $\alpha$.

\begin{definition}
Let $\C{0}$ be the class of all bounded families of metric spaces.
For each ordinal $\alpha > 0$ let 
\[\C{\alpha} = \{ \mathcal{X} \colon \forall R,\ \exists\beta < \alpha,\ \exists \mathcal{Y}\in \C{\beta} \ \mathcal{X} \rdecomp \mathcal{Y}\}.\]
We let
\[
\C{} = \bigcup_\alpha \C{\alpha}\text{ and }\APC = \bigcup \C{}.
\]
We say that $X$ has \df{finite APC-decomposition complexity} if it belongs to $\APC$.
\end{definition}

Given a property of metric spaces one would like to understand its so-called permanence properties; i.e., the extent to which the property is preserved by forming unions, products, etc. While FDC enjoys very strong permanence properties~\cite{Guentner-Permanence}, permanence properties for APC are more elusive. Indeed, only recently, some 16 years after APC first appeared, was it shown that APC is preserved by direct products~\cite{bn-stability,davila} and free products~\cite{bn-stability}.

In~\cite{dydak2016}, Dydak considered decomposition complexity with respect to several coarse properties, including APC. There it was shown that this class is closed under finite unions and some natural types of infinite unions. It was also shown that spaces in a collection in $\C{\alpha}$ have property A and satisfy so-called limit permanence. Our definition in terms of ordinals provides more control on the decompositions and so leads to statements involving upper bounds on the depth $\alpha$ of the decomposition complexity.

Note that if $X$ has FDC, then $\{ X \} \in \C{}$. 
Theorem~\ref{thm:decomposition A} implies that if $\mathcal{X} \in \C{}$, then $\mathcal{X}$ has uniform property~A.


We show that the class $\APC$ is closed under many group operations. These results are summarized in the theorem below.

\begin{theorem}
  Let $H$ and $K$ be countable groups with proper left-invariant metrics. If $H, K \in \APC$ then 
\begin{enumerate}
  \item $H\times K\in \APC$;
  \item $H\ast_C K\in\APC$, where $C$ is some common subgroup;
  \item $G\in\APC$ where $1\to K\to G\to H\to 1$ is exact; and 
  \item $H\wr K\in\APC$.  
\end{enumerate} 
\end{theorem}

We leave it as an open question whether Theorem~\ref{thm:decomposition C} is still true if we change the order of quantifiers to the order used in Theorem~\ref{thm:decomposition A}. If this were true, it would show that if $\mathcal{X} \in \C{\alpha}$ for some ordinal $\alpha$, then $\mathcal{X}$ has uniform asymptotic property~C. This would show that finite decomposition complexity implies APC.

\section{Preliminaries}

\subsection{Metric families}

In this paper we are concerned with applying properties of coarse geometry to metric spaces. Often, these properties will need to be applied in some \emph{uniform} way to a family of spaces. It will therefore be convenient to define coarse geometric notions for families of metric spaces.

We begin by describing the terms \emph{uniformly expansive} and \emph{effectively proper} as they apply to maps $F$ between families of metric spaces. There are several different approaches to describing uniform properties of families of metric spaces. We follow the approach of Guentner~\cite{Guentner-Permanence}.

\begin{definition}
Let $\mathcal{X}$ and $\mathcal{Y}$ be families of metric spaces. A map $F:\mathcal{X}\to \mathcal{Y}$ is a collection of maps $f:X_f\to Y_f$ with $\{X_f\}_{f\in F}=\mathcal{X}$ and $Y_f\in\mathcal{Y}$. We use the notation $F^{-1}(\mathcal{Y})$ to denote the set $\{f^{-1}(Y)\colon f\in F, Y\in\mathcal{Y}\}$. We say that $F$ is \df{uniformly expansive} if there is some non-decreasing $\rho_2:[0,+\infty)\to [0,+\infty)$ such that for every $f\in F$, and for every pair of points $x,x'\in X_f$,  \[d(f(x),f(x'))\le \rho_2(d(x,x'))\hbox{.}\]
We say that $F:\mathcal{X}\to\mathcal{Y}$ is \df{effectively proper} if there is some proper non-decreasing $\rho_1:[0,+\infty)\to [0,+\infty)$ such that for every $f\in F$ and for each pair of points $x, x'\in X_f$, \[\rho_1(d(x,x'))\le d(f(x),f(x'))\text{.}\] We call $F$ a \df{coarse embedding} if it is uniformly expansive and effectively proper. We say that the metric spaces $X$ and $Y$ are \df{coarsely equivalent} if there is a coarse embedding $F:\{X\}\to \{Y\}$ and a positive number $C$ so that if $y\in Y$ then there is some $x\in X$ such that $d_Y(f(x),y)<C$.
\end{definition}

\begin{definition}
Let $\mathcal{X}$ and $\mathcal{Y}$ be families of metric spaces. We write  $\mathcal{X} \prec \mathcal{Y}$ to mean that for every $X\in\mathcal{X}$ there is some $Y\in\mathcal{Y}$ such that $X\subset Y$. 
\end{definition}

\subsection{Groups as metric spaces}

Coarse geometry can often be fruitfully applied to discrete groups. We recall that a metric on a set is said to be \df{proper} if closed balls are compact. A metric on a group is called \df{left-invariant} if the action of the group on itself by left multiplication is an isometry. Finitely generated groups carry a unique (up to coarse equivalence) left-invariant proper metric called the \df{word metric}, which is given by fixing a finite symmetric generating set $S$ and taking the distance $d_S(g,h)$ between the group elements $g$ and $h$ to be the length of the shortest $S$-word that presents the element $g^{-1}h$. Here, we adopt the convention that the neutral element is represented by the empty word.

For a countable (discrete) group $G$ that is not finitely generated, Dranishnikov and Smith showed that up to coarse equivalence $G$ carries a unique proper left-invariant metric~\cite{dr-s06}. This metric is given by taking a (countably) infinite symmetric generating set $S$ and computing a weighted word metric in which the weight function associated to the elements of the infinite generating set $S$ is proper; i.e., for any $N\in \mathbb{N}$ the set of $s\in S$ with weight at most $N$ is finite. 

Thus, if we restrict our attention to proper left-invariant metrics, then the coarse geometric properties of any countable group are group properties. Whenever we consider such countable groups, we will always assume them to have proper left-invariant metrics.

\section{A Decomposition theorem for $\C{\alpha}$}

The goal of this section is to prove a decomposition theorem for $\C{\alpha}$. To begin, we show that $\C{\alpha}$ is a coarse invariant.

\begin{theorem}[Coarse Invariance]\label{thm:coarse-invariant}
The property $\C{\alpha}$ is a coarse invariant. More precisely, if $F:\mathcal{X}\to\mathcal{Y}$ is a coarse embedding and $\mathcal{Y}\in\C{\alpha}$, then $\mathcal{X}\in\C{\alpha}$.
\end{theorem}

\begin{proof} Let $\rho_1$ and $\rho_2$ denote the control functions for the collection of maps $F$.
Suppose first that $\mathcal{Y}\in\C{0}$; i.e., there is some $B>0$ that is a (uniform) bound on the diameters of elements of $\mathcal{Y}$. Given $X\in\mathcal{X}$ there is some $Y\in\mathcal{Y}$ and an $f_X\in F$ so that $f_X^{-1}(Y)=X$. Since $f_X$ is effectively proper, if $x,x'\in X$, then $\rho_1(d(x,x'))\le d(f_X(x),f_X(x'))\le B$. Since $\rho_1$ is proper, there is some $B'>0$ so that $\rho_1(t)\le B$ implies $t\le B'$. Thus, $d(x,x')\le B'$. Since $B'$ only depends on the uniform $B$ and the uniform $\rho_1$ (and is independent of $Y$), we see that $\mathcal{X}\in\C{0}$.

Now, suppose $\alpha>0$ is an ordinal number and that the theorem holds for all $\gamma<\alpha$. Suppose that some $R\in\mathbb{R}^{\mathbb{N}}$ is given and consider the sequence $S$ given by $S_i=\rho_2(R_i)$. Using $S$ we find some $\beta<\alpha$ and some $\mathcal{Z}\in\C{\beta}$ so that $\mathcal{Y}\xrightarrow{S}\mathcal{Z}$. Put $F_\mathcal{Z}=\{f|_{f^{-1}(Z)}\colon f\in F, Z\in\mathcal{Z}\}$. Then, $F_{\mathcal{Z}}:F^{-1}(\mathcal{Z})\to\mathcal{Z}$ is a coarse embedding with control functions $\rho_1$ and $\rho_2$.

We claim that $\mathcal{X}\rdecomp F^{-1}(\mathcal{Z})$. To this end, we take a positive integer $k$ from the uniform $S$-decomposition $\mathcal{Y}\xrightarrow{S}\mathcal{Z}$.
Let $X\in\mathcal{X}$ be arbitrary. Find $Y\in\mathcal{Y}$ and $f_X\in F$ so that $f_X^{-1}(Y)=X$. Take  $\mathcal{V}_1,\ldots, \mathcal{V}_k$ in $\mathcal{Z}$ such that each $\mathcal{V}_j$ is $S_j$-disjoint and such that the union $\bigcup_j\mathcal{V}_j$ covers $Y$. Then, put $\U_j=\{f^{-1}_X(V)\colon V\in \mathcal{V}_j\}$. Then, $\mathcal{U}_j\subset F^{-1}(\mathcal{Z})$. Clearly, the union of the $\mathcal{U}_j$ covers $X$. Finally, given $U\neq U'$ in some $\mathcal{U}_j$, we take $u\in U$ and $u'\in U'$. Then, $\rho_2(R_j)=S_j\le d(f(u),f(u'))\le \rho_2(d(u,u'))$. Since $\rho_2$ is non-decreasing, we see that each $\mathcal{U}_j$ is $R_j$-disjoint as required. We are done.
\end{proof}

\begin{corollary}\cite[Corollary 10.2]{dydak2016} 
The property of a metric family having finite APC-decomposition complexity is a coarse invariant.
\end{corollary}

Next, we record three simple facts in a lemma; these will be needed in the proof of the decomposition theorem for uniform asymptotic property C.

\begin{lemma} \label{lem:basic facts} Let $\mathcal{X}$ and $\mathcal{Y}$ be families of metric spaces. Let $\alpha$ and $\beta$ be ordinal numbers.
\begin{enumerate}
\item If $\mathcal{X} \in \C{\beta}$ and $\beta < \alpha$, then $\mathcal{X} \in \C{\alpha}$.
\item  If $\mathcal{X}, \mathcal{Y} \in \C{\alpha}$, then $\mathcal{X} \cup \mathcal{Y} \in \C{\alpha}$.
\item  If $\mathcal{X} \prec \mathcal{Y}$ and $\mathcal{Y} \in \C{\alpha}$, then $\mathcal{X} \in \C{\alpha}$.
  \end{enumerate}
\end{lemma}
\begin{proof} Statement (1) follows from the definition. Statement (2) is an easy consequence of the definition. Statement (3) follows from coarse invariance. 
\end{proof}

\begin{decomposition}\label{thm:decomposition}
 Let $\alpha>0$ be an ordinal number. If $\mathcal{Y} \in \C{\alpha}$ and for each $R \in \mathbb{R}^\mathbb{N}$ we have
  $\mathcal{X} \rdecomp \mathcal{Y}$, then $\mathcal{X} \in \C{\alpha}$.
\end{decomposition}
\begin{proof}
For each $R \in \mathbb{R}^\mathbb{N}$ let $\tilde R \in \mathbb{R}^{\mathbb{N}^2}$ be the rearrangement shown in Figure~\ref{fig:rearrangement} (any fixed rearrangement would work).
\begin{figure}
\begin{multicols}{2}
\[
\begin{array}{ccccc}
\vdots & \vdots & \vdots & \vdots & \reflectbox{$\ddots$} \\
\tilde R_{4,1} & \tilde R_{4,2} & \tilde R_{4,3} & \tilde R_{4,4} & \ldots\\
\tilde R_{3,1} & \tilde R_{3,2} & \tilde R_{3,3} & \tilde R_{3,4} & \ldots\\
\tilde R_{2,1} & \tilde R_{2,2} & \tilde R_{2,3} & \tilde R_{2,4} & \ldots\\
\tilde R_{1,1} & \tilde R_{1,2} & \tilde R_{1,3} & \tilde R_{1,4} & \ldots\\
\end{array}
\]
\break
\[
\begin{array}{ccccc}
\vdots & \vdots & \vdots & \vdots & \reflectbox{$\ddots$} \\
R_{10} & R_{14} & R_{19} & R_{25} & \ldots\\
R_6 & R_9 & R_{13} & R_{18} & \ldots\\
R_3 & R_5 & R_8 & R_{12} & \ldots\\
R_1 & R_2 & R_4 & R_7 & \ldots\\
\end{array}
\]
\end{multicols}
\caption{A rearrangement of the sequence $R$ into a $2$-dimensional array.}
\label{fig:rearrangement}
\end{figure}
We will prove that for each $R \in \mathbb{R}^\mathbb{N}$
  there exists a $\beta < \alpha$, a $\mathcal{Z} \in \C{\beta}$, and a number $m$ such that for each $X \in \mathcal{X}$ there exist at most $m$-many collections $\U_{i,j}$ such that
$\U_{i,j} \subset \mathcal{Z}$,
$\U_{i,j}$ is $\tilde R_{i,j}$-disjoint, and
$\bigcup_{i,j} \U_{i,j}$ covers~$X$.
This means that $\mathcal{X} \rdecomp \mathcal{Z}$ and implies that $\mathcal{X} \in \C{\alpha}$.

Let some $R\in\mathbb{R}^\mathbb{N}$ be given and form $\tilde R$. For each pair of natural numbers $(i,j)$, put $S_i^j=\tilde R_{i,j}$. Then, for each fixed $j$, we consider the sequence $S^j=(S^j_i)_{i\ge 1}$ that corresponds to the $j$-th column of $\tilde R$. Without loss of generality, we may assume $S^j$ to be increasing for each $j$.

By the assumption that $\mathcal{Y}\in \C{\alpha}$, for each $j$ there is an ordinal $\beta_j<\alpha$, a family $\mathcal{Z}_j\in\C{\beta_j}$ and an integer $k_j\ge 1$ such that $\mathcal{Y}$ uniformly $S_j$-decomposes over $\mathcal{Z}_j$ into $k_j$-many families satisfying the conditions of Definition~\ref{def:rdecomp}.

Next, for each $j$, put $P_j=S^j_{k_j}$ and consider the sequence $P=(P_j)_{j\ge 1}$. By the assumptions, there is a uniform $P$-decomposition of $\mathcal{X}$ over $\mathcal{Y}$. Take the number $k$ from the uniform $P$-decomposition. Put $\beta=\max\{\beta_1,\ldots,\beta_k\}$ and observe that $\beta<\alpha$. Finally, set $\mathcal{Z}=\bigcup_{j=1}^k\mathcal{Z}_j$.

Now, fix an arbitrary $X\in\mathcal{X}$ and find families $\mathcal{U}_1, \mathcal{U}_2,\ldots,\mathcal{U}_k$ in $\mathcal{Y}$ such that $\mathcal{U}_j$ is $P_j$-disjoint and so that $\bigcup_j \mathcal{U}_j$ covers $X$.

Write the elements of $\mathcal{U}_j$ as $\mathcal{U}_j=\{U^j_\ell\}_{\ell}$. Since we have $\U_j \xrightarrow{S^j} \mathcal{Z}_j \in \C{\beta_j}$ with $k_j$ families, for each $\ell$ we have a sequence of families $\V^{j,\ell}_1, \V^{j,\ell}_2, \ldots, \V^{j,\ell}_{k_j}$
such that $\V^{j,\ell}_i \subset \mathcal{Z}_j$, $\V^{j,\ell}_i$ is $\tilde R_{i,j}$-disjoint and $\bigcup_i \V^{j,\ell}_i$ covers $U^j_\ell$.

Put $\U_{i,j} = \bigcup_\ell \V^{j,\ell}_i$. We see that from Lemma~\ref{lem:basic facts}, that $\mathcal{Z}$ is in $\C{\beta}$ and note that $\U_{i,j}\subset \mathcal{Z}$. 

Next, $\bigcup_{i,j} \U_{i,j}$ covers $X$ since $\bigcup_i \U_{i,j}$ covers $\bigcup \U_j$ for each $j$ and $\bigcup_j \U_j$ covers~$X$.

Finally, we show that each $\U_{i,j}$ is $\tilde R_{i,j}$-disjoint. We know that each family $\mathcal{V}_i^{j,\ell}$ is $\tilde R_{i,j}$-disjoint. Since $\U_{i}$ is $\tilde R_{k_j,j}$-disjoint, we know that when $\ell_1\neq \ell_2$, then $U^j_{\ell_1}$ and $U^j_{\ell_2}$ are at least $\tilde R_{k_j,j}$ apart. Because $\tilde R_{k_j,j}\ge \tilde R_{i,j}$ whenever $i\in\{1,2,\ldots, k_j\}$, we conclude that $\mathcal{U}_{i,j}$ is $\tilde R_{i,j}$-disjoint for $j=1,2,\ldots, k$ and $i=1,2,\ldots, k_j$.

To complete the proof, we unravel the indexing $\tilde R_{i,j}$ back to the original $R_i$ and observe that we have at most $m:= k\cdot \max\{k_1,\ldots, k_k\}$ families, where we fill in any gaps in the sequence with empty families as necessary. Note that the value of $m$ depends only on $R$ and not on the space $X$.
\end{proof}

\begin{lemma}\label{lem:apc is c1}
A family $\mathcal{X}$ of metric spaces
has uniform asymptotic property~C
if and only if
$\mathcal{X} \in \C{1}$.
\end{lemma}
\begin{proof} The family $\mathcal{X}$ is said to have uniform asymptotic property~C if for every sequence $R\in\mathbb{R}^{\mathbb{N}}$ there is a bounded family $\mathcal{Y}_R$ of metric spaces with the property that $\mathcal{X}\rdecomp\mathcal{Y}$. But, $\mathcal{Y}_R$ is bounded if and only if $\mathcal{Y}_R\in\C{0}$. We observe that $\mathcal{X}\in\C{1}$ if and only if for every $R$ there is some ordinal $\alpha<1$ and some $\mathcal{Y}\in \C{\alpha}$ so that $\mathcal{X}\rdecomp \mathcal{Y}$.
\end{proof}

\begin{proof}[Proof of Theorem~\ref{thm:decomposition C}]
  Apply the Decomposition Theorem for $\alpha = 1$ and then apply  Lemma~\ref{lem:apc is c1}.
\end{proof}

When we compare $\APC$ to FDC, we see that $\mathcal{X}$ has FDC if and only if $\mathcal{X}$ is in $\mathfrak{D}_\alpha$ for a countable ordinal $\alpha$~\cite[Theorem 2.2.2]{Guentner-Tessera-Yu-2013}. We have a similar result for our finite APC-decomposition complexity. 

\begin{theorem} A metric family has finite APC-decomposition complexity precisely when it belongs to $\C{\alpha}$ for some countable ordinal $\alpha$; i.e., $\C{}=\C{\omega_1}$.
\end{theorem}

This follows from the following lemma.

\begin{lemma}
If $\mathcal{X}\in \C{\alpha}$ for some $\alpha>0$, then there is a collection $\{\mathcal{Y}_m\}_{m=1}^\infty$ such that for every $R\in \mathbb{R}^{\mathbb{N}}$, there is an $n$ such that $\mathcal{X}\xrightarrow{R}\mathcal{Y}_n$ and $\mathcal{Y}_n\in\C{\beta}$ for some $\beta<\alpha$. 
\end{lemma}

\begin{proof} It suffices to consider sequences $R$ of natural numbers. 
Given such a sequence $R=\left(n_1,n_2,\ldots\right)$, take $\bar{R}$ to be the shortest initial segment $(n_1,\ldots, n_k)$ so that for every $X\in\mathcal{X}$ there is a $\mathcal{Y}$ and subsets $\mathcal{U}_i$ of $\mathcal{Y}$, ($i=1,\ldots,k$) such that $\mathcal{U}_i$ is $n_i$-disjoint and so that $\bigcup_{i=1}^k\mathcal{U}_i$ covers $X$. 

Define an equivalence relation on sequences of natural numbers by $R_1\sim R_2$ if and only if $\bar{R_1}=\bar{R_2}$. We note that there are countably many such equivalence classes. For each class $[R]$, we choose some $\mathcal{Y}$ in $\C{\beta}$ for some $\beta<\alpha$ such that if $P\in[R]$, then $\mathcal{X}\xrightarrow{P}\mathcal{Y}$. 
\end{proof}

\section{Permanence Theorems for $\APC$}

\subsection{Permanence results for spaces}

In Guentner's survey~\cite{Guentner-Permanence}, he describes four so-called ``primitive permanence results'' for spaces: coarse invariance, union permanence, fibering permanence, and limit permanence. We establish these results before stating the ``derived permanence results'' that follow from the primitive results. 

We have already seen that $\C{\alpha}$ is a coarse invariant in Theorem~\ref{thm:coarse-invariant}.

To simplify matters, we state the union theorem in terms of a single metric space instead of a metric family. 

\begin{theorem}\label{thm:GTY-union} Let $X$ be a metric space that is expressed as a union $X=\bigcup_{i\in I}X_i$. Suppose further that $\{X_i\}_i\in\C{\alpha}$ and that for every $r>0$ there is a subspace $Y=Y(r)\subset X$ with $\{Y(r)\}\in\C{\alpha}$ such that the collection $\{X_i-Y\}$ is $r$-disjoint. Then, $\{X\}\in\C{\alpha+1}$. In particular, finite APC-decomposition complexity is preserved by so-called excisive unions.
\end{theorem}

\begin{proof}
For a given $R\in\mathbb{R}^{\mathbb{N}}$, take $r=R_1$. Consider the collection $\{X_i-Y(r)\}_{i\in I}$. Since $\{X_i\}\in \C{\alpha}$, by Lemma~\ref{lem:basic facts}, we see that $\{X_i-Y(r)\}\in\C{\alpha}$. With $\U_{1}=\{X_i-Y(r)\}_{i\in I}$ and $\U_{2}=\{Y(r)\}$, we again apply Lemma~\ref{lem:basic facts} to see that $\{X\}$ is uniformly $R$-decomposable over a family from $\C{\alpha}$.
\end{proof}

\begin{corollary}\cite[Proposition 6.5]{dydak2016}
If $X$ and $Y$ have finite APC-decomposition complexity, then $X\cup Y$ has finite APC-decomposition complexity.
\end{corollary}

We can improve Theorem~\ref{thm:GTY-union} by applying the Decomposition Theorem to show that $\{X\}$ actually belongs to $\C{\alpha}$.

\begin{theorem} [Union Permanence] Let $X$ be a metric space with $X=\bigcup_{i\in I}X_i$. Suppose that $\{X_i\}_i\in\C{\alpha}$ and for every $r>0$ there is a subspace $Y(r)\subset X$ so that $\{Y(r)\}_{r>0}\in\C{\alpha}$ and the collection $\{X_i-Y\}_i$ is $r$-disjoint. Then, $X\in\C{\alpha}$.
\end{theorem}

\begin{proof}
To apply the Decomposition Theorem, we have to find a family $\mathcal{Y}$ so that for every $R$, we have $\mathcal{X}\xrightarrow{R}\mathcal{Y}$. 

Put $\mathcal{Y}=\{Y(r)\}_{r>0}\cup \bigcup_{r>0,i}\{X_i-Y(r)\}$. Given $R$, we take $\U_{1}=\{X_i-Y(R_1)\}_{i\in I}$ and $\U_{2}=\{Y(R_1)\}$. as in the proof of Theorem~\ref{thm:GTY-union}, we apply (2) and (3) from Lemma~\ref{lem:basic facts} to see that $\{X\}$ is uniformly $R$-decomposable over $\mathcal{Y}$, which is in $\C{\alpha}$.
\end{proof}

Before proving Fibering Permanence, we consider the case of direct products. We consider this separately to compare with the product theorem for APC~\cite{bn-stability}.

\begin{definition}
  For any collections $\mathcal{X}, \mathcal{Y}$ we let
  \[
    \mathcal{X} \otimes \mathcal{Y} = \{ X \times Y \colon X \in \mathcal{X}, Y \in \mathcal{Y} \}.
  \]
  We give spaces in $\mathcal{X}\otimes\mathcal{Y}$ the product metric: \[d_{X\times Y}(x\times y, x'\times y')=\sqrt{d_X(x,x')^2+d_Y(y,y')^2\text{.}}\]
\end{definition}

\begin{product} Let $\mathcal{X}$ and $\mathcal{Y}$ be in $\C{\alpha}$. Then, $\mathcal{X}\otimes\mathcal{Y}\in\C{\alpha}$.
\end{product}

\begin{proof}
We proceed by induction. For $\alpha=0$ this is obvious.

We use the technique and much of the notation from the Decomposition Theorem.

As in that proof, we rearrange a given sequence into a $2$-dimensional array as shown in Figure~\ref{fig:rearrangement}. We find decompositions of $\mathcal{X}$ over the columns $S^j$ of this array. We find $\mathcal{X}_j$ such that $\mathcal{X}\xrightarrow{S^j}\mathcal{X}_j$. Next, we use the sequence $P_j$ to find a decomposition of $\mathcal{Y}$ over some family $\mathcal{Y}^\ast$. Let $m$ be the number of families needed for such a decomposition.

Now, for any $X\times Y\in\mathcal{X}\otimes\mathcal{Y}$, the construction in the proof of~\cite[Theorem 3.1]{bn-stability} provides at most $k$ families $\mathcal{W}_j$ such that $\mathcal{W}_j$ is $R_j$-disjoint, $\bigcup_j \mathcal{W}_j$ covers $X\times Y$, and $\mathcal{W}_j\subset \bigcup_{i=1}^m \mathcal{X}_i\otimes \mathcal{Y}^\ast$. By the inductive assumption and Lemma~\ref{lem:basic facts}(2), we see that this union is in $\C{\gamma}$ with $\gamma<\alpha$.

\end{proof}

\begin{lemma} \label{lem:pre-fiber}
Let $F:\mathcal{X}\to \mathcal{Y}$ be a uniformly expansive map of metric families $\mathcal{X}$ and $\mathcal{Y}$. Let $\alpha$ be an ordinal and suppose $\mathcal{Y}'\prec\mathcal{Y}$  
Then the following are equivalent:
\begin{enumerate}
	\item if $\mathcal{Y}'\in\C{0}$, then $f^{-1}(\mathcal{Y}')\in\C{\alpha}$; and
    \item for any ordinal $\beta$ if $\mathcal{Y}'\in\C{\beta}$, it follows that $f^{-1}(\mathcal{Y}')\in\C{\alpha+\beta}$.
\end{enumerate}
\end{lemma}

\begin{proof}
It will suffice to prove (1) implies (2). We proceed by transfinite induction on $\beta$. 

Assume (1) holds. The case $\beta=0$ is obvious. Thus we may assume that $\beta>0$ and that the situation in (2) holds for every $\gamma<\beta$. Suppose $\mathcal{Y}'\in\C{\beta}$. 
Let $R\in\mathbb{R}^{\mathbb{N}}$ be given and consider the sequence $S$ defined by $S_i=\rho(R_i)$, where $\rho$ is the control function corresponding to the uniformly expansive map $F$. Using the sequence $S$, we find a $\gamma<\beta$ and $\mathcal{Z}\in\C{\gamma}$ so that $\mathcal{Y}'\xrightarrow{S}\mathcal{Z}$. Let $\mathcal{Z}'=\{Z\in\mathcal{Z}\colon \exists Y'\in\mathcal{Y}'\text{ such that } Z\subset Y'\}$. Then $\mathcal{Z}'\prec\mathcal{Y}$. By the induction hypothesis, $F^{-1}(\mathcal{Z})\in \C{\alpha+\gamma}$. To finish the proof, it remains to show that $F^{-1}(\mathcal{Y}')\rdecomp F^{-1}(\mathcal{Z})$. 

To this end, take the integer $k$ from the uniform $S$-decomposition $\mathcal{Y}'\xrightarrow{S}\mathcal{Z}$. Let $W\in F^{-1}(\mathcal{Y}')$; i.e., $W=f^{-1}(A)$ for some $A\in\mathcal{Y}'$ and $f\in F$. Take families $\U_1,\U_2,\ldots, \U_k\subset \mathcal{Z}$ such that $\U_j$ is $S_j$-disjoint and $\bigcup_j \U_j$ covers $A$. Then, put $\mathcal{V}_i=\{f^{-1}(U)\colon U\in \U_j\}$. 
Since each member of $\bigcup_j\mathcal{U}_j$ is a subset of $A$, we have that each of the $\mathcal{V}_j$ is a subset of $F^{-1}(\mathcal{Z}')$. We also have that each $\mathcal{V}_j$ is $R_j$-disjoint and $\bigcup_j\mathcal{V}_j$ covers $W$ as required. 
\end{proof}


\begin{fibering}\label{thm:fibering}
Let $F:\mathcal{X}\to \mathcal{Y}$ be a uniformly expansive map of metric families $\mathcal{X}$ and $\mathcal{Y}$. Suppose that there is some $\alpha$ so that for any bounded family $\mathcal{B}\prec \mathcal{Y}$, the set $F^{-1}(\mathcal{B})\in \C{\alpha}$. Then, if there is some $\beta$ such that $\mathcal{Y}\in\C{\beta}$, then $\mathcal{X}\in\C{\alpha+\beta}$.
\end{fibering}

\begin{proof}
Condition (1) of Lemma~\ref{lem:pre-fiber} is satisfied by $\mathcal{B}$, so condition (2) holds, which is what we needed to show.
\end{proof}

\begin{corollary} Suppose $f:X\to Y$ is a uniformly expansive map of metric spaces, $\{Y\}\in\C{\beta}$ for some $\beta$ and suppose that $f^{-1}(\mathcal{B})\in\C{\alpha}$ for every bounded family $\mathcal{B}$ of subsets of $Y$. Then, $\{X\}\in\C{\alpha+\beta}$. \qed 
\end{corollary}

\begin{limit}
Let $\mathcal{X}=\{X_a\}_{a\in J}$ be a collection of metric spaces indexed by some indexing set $J$. Suppose that for every real number $r>0$, there is an expression $X_a=\bigcup_i X^i_a$ as an $r$-disjoint union such that for each $a$, the family $\{X^i_a\}_{i,a}$ belongs to $\C{\alpha}$. Then $\mathcal{X}$ belongs to $\C{\alpha+1}$.  
\end{limit}

\begin{proof} 
Let $R\in\mathbb{R}^{\mathbb{N}}$ be given. Let $X_a\in\mathcal{X}$. Write $X_a=\bigcup X_a^i$ as an $R_1$-disjoint union over the family $\{X^a_i\}$ with $\{X^i_a\}_i\in\C{\alpha}$. 
\end{proof}

This is useful in that it shows that $\APC$ satisfies limit permanence.

\begin{corollary} \cite[Corollary 11.3]{dydak2016} Finite APC-decomposition complexity satisfies limit permanence. \qed 
\end{corollary}

\subsection{Permanence results for groups}

Having proven the primitive permanence results, we turn our attention to derived results. Most of these are interesting in the context of groups. Throughout this section, all groups are assumed to be countable discrete groups in left-invariant proper metrics.

\begin{theorem}\label{thm:short-exact-sequences}
Let
\[
  1 \to K \to G \to H \to 1
\]
be a short exact sequence of groups.
If $K$ and $H$ have finite APC-decomposition complexity, then $G$ has finite APC-decomposition complexity.
\end{theorem}

\begin{proof} As stated, this follows from the fact that any coarse property that satisfies subspace, finite union, and fibering permanence is closed under group extensions~\cite[Corollary 7.5]{Guentner-Permanence}.

We can apply our Fibering Permanence to prove the stronger result that if $\{H\}\in\C{\alpha}$ and $\{K\}\in\C{\beta}$, then $\{G\}\in\C{\alpha+\beta}$. 

To see this, we fix a proper left-invariant metric on $H$ arising from a weighting of a generating set $T$. Let the surjective homomorphism from $G$ to $H$ be denoted by $\varphi$ so that $K=\ker\varphi$. For each $t_i\in T$, take some $s_i\in G$ such that $\varphi(s_i)=t_i$. Give each $s_i$ the same weight as was assigned to $t_i$ and adjoin elements to the collection $\{s_i\}$ to form a generating set $S$ for $G$. Take a weighting function on $S$ that extends the values on the $s_i$ and consider $G$ in this left-invariant proper metric. Note, that this metric is unique up to coarse equivalence. In these metrics $\phi$ is $1$-Lipschitz and so it is uniformly expansive. Next, take the metric on $K$ that it inherits as a subgroup of $G$. If $B\subset H$ is any bounded subset, then $\phi^{-1}(B)$ is contained in a neighborhood of $K$ (see, for example~\cite[Theorem 3]{asdim-groups}) and so $\{\phi^{-1}(B)\}$ has $\C{\beta}$ because $\{K\}$ does. By Fibering Permanence we are done. 
\end{proof}

\begin{theorem}
Let $G$ be a countable discrete group acting (without inversion) on a tree in such a way that the vertex stabilizers have finite APC-decomposition complexity. Then, $G$ has finite APC-decomposition complexity.
\end{theorem}

\begin{proof} This follows from Guentner's work~\cite[Theorem 7.6]{Guentner-Permanence} as a consequence of subspace permanence, coarse invariance, union permanence, and fibering permanence. 
\end{proof}

The next result follows from the Bass-Serre theory~\cite{BH,Serre}.

\begin{corollary}
Let $A$ and $B$ be countable (discrete) groups in proper left-invariant metrics. If $A$ and $B$ are in $\APC$, then the amalgamated free product $A\ast_C B$ and the HNN extension $A*_C$ are in $\APC$. 
\end{corollary}

\begin{theorem}
Let $G$ and $H$ be countable (discrete) groups in proper left-invariant metrics. If $G$ and $H$ are in $\APC$ then $H\wr G\in\APC$. 
\end{theorem}

\begin{proof} We give the sketch of the general argument from \cite[Theorem 8.4]{Guentner-Permanence}.

By Fibering Permanence, finite sums of $H$ have $\APC$. By Limit Permanence, the group $H^{(G)}$ of finitely supported $H$-valued functions on $G$ on which $G$ acts by translations is in $\APC$. Then, one last application of Fibering Permanence shows that $H\wr G$ is in $\APC$. 
\end{proof}

\section{Finite APC Complexity and Property A}

In this section we show that spaces with finite APC-decomposition complexity have property A; we also show a decomposition theorem for uniform property A.

The notion of $\epsilon$-variation was defined in the introduction. The following lemma is trival.
%

\begin{lemma}\label{lem:sum}
  If $\xi_i$ has $\varepsilon_i$-variation, then $\sum \xi_i$ has $\sum \varepsilon_i$-variation.
\end{lemma}

\begin{definition}
  We say that a map $\xi \colon X \to \ell^1(X)$ is \df{normed} if $\| \xi_x \|_1 = 1$ for each $x \in X$. Such a $\xi$ is said to be $S$-\df{locally supported} (for $S>0$) if $\supp(\xi_x)\subset \bar B(x,S)$ for all $x\in X$. We call $\xi$ \df{locally supported} it if is $S$-locally supported for some $S$. \end{definition}

\begin{lemma}\label{lem:normed}
  If $\xi$ has $\varepsilon$-variation and for each $x \in X$ we have $\| \xi_x \|_1 \geq 1$, then $\bar \xi$ defined by the formula $\bar \xi_x = \xi_x / \| \xi_x \|_1$ is normed and has $2\varepsilon$-variation.
\end{lemma}

\begin{proof}
For any maps nonzero maps $u$ and $v$ in $\ell^1(X)$, we have \[\left\|\frac{u}{\|u\|_1}-\frac{v}{\|v\|_1}\right\|_1\le \frac{1}{\|u\|_1}\|u-v\|_1+\left\|\frac{\|v\|_1 v-\|u\|_1v}{\|u\|_1\|v\|_1}\right\|_1\le \frac{2}{\|u\|_1}\|u-v\|_1\hbox{.}\]
The result easily follows.
\end{proof}

\begin{lemma} \label{lem:loc-supp}
  Let $U \subset X$.
  Let $\xi \colon U \to \ell^1(U)$ be a normed, locally supported map with $\varepsilon$-variation.
  Let $R \in \mathbb{N}$. Let $N(U,R)$ denote the open $R$-neighborhood of $U$ in $X$.
  There exists a locally supported map $\bar \xi \colon X \to \ell^1(X)$ such that
\begin{enumerate}
\item $\| \bar \xi_x \|_1 = 0$ for $x \in X$ such that $d(x, U) \geq R$.
\item $\bar \xi_x = \xi_x$ for $x \in U$.
\item $\bar \xi$ has $((2R+1)\varepsilon + \frac 1R)$-variation.
\end{enumerate}
\end{lemma}
\begin{proof}
Define $\eta \colon X \to [0, 1]$ by the formula
\[
\eta(x) =
\left\{
\begin{array}{ll}
1 & x \in U \\
\frac1R d(x, X \setminus N(U, R)) & x \not\in U.
\end{array}
\right.
\]
The map $\eta$ is $\frac1R$-Lipschitz. 

For $x \in X$ let 
$u(x)$ be $x$ if $x \in U$; any point in $U$ such that $d(u(x), x) \leq R$ if $x \in N(U, R)$; any point in $U$ otherwise.
We let
\[
\bar \xi_x = \eta(x) \cdot \xi_{u(x)}
\]

Let $x_1,x_2 \in X$ with $d(x_1, x_2) \leq k$. We have
\[
  \| \bar \xi_{x_1} - \bar \xi_{x_2} \|_1
=
  \| \eta(x_1) \cdot \xi_{u(x_1)} - \eta(x_2) \cdot \xi_{u(x_2)} \|_1
\]

We apply the inequality
\[
\| au - bv \|_1 \leq a\|u-v\|_1 + |a-b| \|v\|_1
\]
to get 
\[
  \| \eta(x_1) \cdot \xi_{u(x_1)} - \eta(x_2) \cdot \xi_{u(x_2)} \|_1
\leq
\eta(x_1) \| \xi_{u(x_1)} - \xi_{u(x_2)} \|_1
+ |\eta(x_1)-\eta(x_2)|\|\xi_{u(x_2)}\| _1\text{.}
\]
If $x_1 \in X \setminus N(U,R)$, then $\eta(x_1) = 0$ and
\[
\| \bar \xi_{x_1} - \bar \xi_{x_2} \|_1
\leq 
|\eta(x_1) - \eta(x_2)| \leq \frac{k}{R} \]
since $\eta$ is $\frac1R$-Lipschitz. 

If $x_2 \in X \setminus N(U,R)$, then the situation is analogous.

If $x_1,x_2 \in N(U,R)$, then
\[
d(u(x_1), u(x_2)) \leq R + d(x_1,x_2) + R \leq 2R + k.
\]
Since $\xi$ has $\varepsilon$-variation, we have
\[
\| \bar \xi_{x_1} - \bar \xi_{x_2} \|_1
\leq 
(2R+k)\varepsilon + \frac kR \leq
k\left((2R+1)\varepsilon + \frac 1R\right)
\]
Combining both cases we have
\[
\| \bar \xi_{x_1} - \bar \xi_{x_2} \|_1
\leq 
k\left((2R+1)\varepsilon + \frac 1R\right)
\]
for each $x_1,x_2 \in X$ with $d(x_1,x_2) \leq k$. Hence $\bar \xi$ has $((2R+1)\varepsilon + \frac 1R)$-variation. If $\xi$ is $S$-locally supported, then it follows from the construction that $\bar\xi$ is $(R+S)$-locally supported.
\end{proof}

\begin{lemma}
Let $X$ be a metric space.
Let $\U_i$ be a finite sequence of $R_i$-disjoint families of subsets of $X$ such that $\bigcup_i \U_i$ covers $X$.
For $U \in \U_i$, let $\xi^U \colon U \to \ell^1(U)$ be a normed, locally supported map with $\varepsilon_i$-variation.
Then there exists a normed, locally supported map $\xi \colon X \to \ell^1(X)$
  with $E$-variation,
  where
  \[
  E=2\sum_{i}\left((2R_i+1)\varepsilon_i + \frac 1{R_i}\right).
  \]
\end{lemma}
\begin{proof} This follows immediately from Lemmas~\ref{lem:sum}, \ref{lem:normed}, and \ref{lem:loc-supp}.
\end{proof}

\begin{corollary}\label{cor:unif-A}
  Let $X$ be a metric space.
  Let $\varepsilon = \frac 1N > 0$.
  Let $R_i = 2^{i+1} N$ and $\varepsilon_i = \frac 1{4^{i+2} N}$.
  Let $\U_i$ be a finite sequence of $R_i$-disjoint families of subsets of $X$ such that $\bigcup \U_i$ covers $X$.
  If for each $U \in \U_i$ there exists a locally supported normed map $\xi_U \colon U \to \ell^1(U)$ with $\varepsilon_i$-variation, then there exists a locally supported normed map $\xi \colon X \to \ell^1(X)$ with $\epsilon$-variation. 
\end{corollary}

We are now in a position to prove our decomposition theorem for uniform property A from the introduction. 

\begin{proof}[Proof of Theorem~\ref{thm:decomposition A}]
Suppose that $\mathcal{X}$ has the property that for every $R\in\mathbb{R}^{\mathbb{N}}$ there is a family $\mathcal{Y}_R$ of metric spaces with uniform property A such that $\mathcal{X}$ admits a uniform $R$-decomposition over $\mathcal{Y}_R$. 

Let $\epsilon>0$ be given; we may assume $\epsilon$ to be of the form $\frac1n$ for some $n\in\mathbb{N}$. Then, as in Corollary~\ref{cor:unif-A}, take the sequence $R$ given by $R_i=2^{i+1}\frac1\epsilon$ and put $\epsilon_i=\frac{1}{4^{i+2}}\epsilon$. By assumption, we can find a family $\mathcal{Y}_R$ with uniform property A with the property that there is some $k$ so that for any $X\in\mathcal{X}$ there are families $\U_1,\ldots,\U_k$ of subsets from $\mathcal{Y}_R$ whose union covers $X$. Use the uniform assumption with $\epsilon_i$ as above to find $S_i$ and maps $\xi^U_i:U\to \ell^1(U)$ realizing the uniform property A condition. Then, with $S=\max\{S_i\}$, we have a map $\xi:X\to \ell^1(X)$ that is $(R_k+S)$-locally supported, is normed and is of $\epsilon$-variation.
\end{proof}

We can use this result to conclude that spaces with finite APC-decomposition complexity have property A. We remark that this was also shown in \cite[Proposition 11.1]{dydak2016} using different techniques.

\begin{definition}
Let $X$ be a discrete metric space with bounded geometry.
We say that $X$ has \df{property A} if for each $R, \varepsilon > 0$ there
  exists a map $\xi \colon X \to \ell^1(X)$ such that
\begin{enumerate}
\item $\| \xi_x \|_1 = 1$ for all $x \in X$,
\item if $x_1, x_2 \in X$ and $d(x_1,x_2) \leq R$, then $\| \xi_{x_1} - \xi_{x_2} \|_1 \leq \epsilon$,
\item there exists $S > 0$ such that $\supp \xi_x \subset \bar B(x, S)$ for all $x \in X$.
\end{enumerate}
\end{definition}

\begin{theorem}\label{thm:propA}
If $X$ is a discrete metric space with bounded geometry and finite APC-decomposition complexity, then $X$ has property A.
\end{theorem}

\begin{proof}
It is clear that any bounded family has uniform property A. Suppose therefore that $\{X\}\in\C{\alpha}$ for some $\alpha>0$. Then, for any $R$, $X$ admits a uniform $R$-decomposition over some family $\mathcal{Y}\in\C{\beta}$ with $\beta<\alpha$. By the inductive assumption, $\mathcal{Y}$ has uniform property A. By Theorem~\ref{thm:decomposition A}, $\{X\}$ has uniform property A. We conclude that $X$ has property A as desired.
\end{proof}


\section{Open Questions}

We end with several open questions on APC and finite APC-decomposition complexity. 

The first question was stated in the introduction:

\begin{question}
Does finite decomposition complexity imply asymptotic property C?
\end{question}

Several of the results in this paper involve an increase in the depth of the APC-decomposition complexity, $\alpha$. We could ask several questions when $\alpha=1$, which is the case of (uniform) APC.

\begin{question}   Let $H$ and $K$ be countable groups with proper left-invariant metrics and APC. 
\begin{enumerate}
  \item Does $H\ast_C K$ have APC, where $C$ is some nontrivial common subgroup?
  \item Does $G$ have APC where $1\to K\to G\to H\to 1$ is exact?
  \item Do the groups $\bigoplus H$ or $H\wr K$ have APC?
\end{enumerate} 
\end{question}

More generally, we can ask whether the complexity level necessarily increases in the permanence results on fibering or the limit; e.g.:

\begin{question} Let $F:\mathcal{X}\to\mathcal{Y}$ be a uniformly expansive map. 
Suppose that $\mathcal{Y}$ is in $\C{\alpha}$ with $\alpha>0$ and that for every bounded family $\mathcal{B}\prec \mathcal{Y}$ the family $F^{-1}(\mathcal{B})\in\C{\alpha}$. Does it follow that $\mathcal{X}\in\C{\alpha}$? 
\end{question}

\appendix
\section{On finite APC-decomposition complexity and sFDC}

The anonymous referee offered the following suggestion and proof, which we include with our sincere thanks.

\begin{def*}[see~\cite{DrZ14}]
We say that a family $\mathcal{X}$ of metric spaces has \df{straight finite decomposition complexity} if for every $(R_i)_{i\in\mathbb{N}}\in\mathbb{R}^{\mathbb{N}}$ there exists a $k\in \mathbb{N}$ and families $\mathcal{X}_1,\ldots,\mathcal{X}_k$ of metric spaces such that $\mathcal{X}_{i-1}$ is $R_i$-decomposable over $\mathcal{X}_i$ for every $i\in \{1,\dots,k\}$, where $\mathcal{X}_0=\mathcal{X}$, and $\mathcal{X}_k$ is bounded. 
\end{def*}

\begin{thm*}
Every metric space with finite APC-decomposition complexity has straight finite decomposition complexity.
\end{thm*}

 We note that Theorem~\ref{thm:propA} follows from this theorem and \cite[Theorem 4.2]{DrZ17}. 
 
 The proof is similar to \cite[Proposition 3.2]{DrZ14}.
 
\begin{proof}
It suffices to show the following property $(\ast_\alpha)$ holds for every ordinal $\alpha$ by transfinite induction.

$(*_\alpha)$ $\mathcal{X}$ has straight finite decomposition complexity for every $\mathcal{X}\in\C{\alpha}$.

Property $(*_0)$ holds obviously. 

Suppose that $\alpha>0$ and $(*_\beta)$ holds for every $\beta<\alpha$. Let $\mathcal{X}\in\C{\alpha}$ and $R=(R_i)_{i\in\mathbb{N}}\in\mathbb{R}^{\mathbb{N}}$. Take $\beta<\alpha$ and $\mathcal{Y}_0\in\C{\beta}$ so that $\mathcal{X}$ is uniformly $R$-decomposable over $\mathcal{Y}_0$. Let $k$ be the number from the uniform $R$-decomposition. For every $X\in\mathcal{X}$, take $\U_{1}^X,\ldots,\U_{k}^X\subset \mathcal{Y}_0$ such that each $\U_{i}^X$ is $R_i$-disjoint and $\bigcup_{i=1}^k\U_{i}^X$ covers $X$. Set $\mathcal{X}_0^X=\{X\}$ and 
\[\mathcal{X}_i^X=\bigcup_{j=1}^i\left\{ U\setminus\bigcup\bigcup_{k=1}^{j-1}\U_{k}^X\colon U\in\U_{j}^X\right\}\cup\left\{X\setminus \bigcup\bigcup_{j=1}^i\U_{j}^X\right\}\]
for $i\in\{1,\ldots,k\}$. Let $\mathcal{X}_i=\bigcup_{X\in \mathcal{X}}\mathcal{X}_i^X$ for $i\in\{0,1,\ldots, k\}$. Then $\mathcal{X}_{i-1}$ is $R_i$-decomposable over $\mathcal{X}_i$ for every $i\in\{1,\ldots,k\}$, $\mathcal{X}_0=\mathcal{X}$, and $\mathcal{X}_k\prec\mathcal{Y}_0$. 

By the induction assumption, $\mathcal{Y}_0$ has straight finite decomposition complexity. Thus there exist $m\in\mathbb{N}$ and families $\mathcal{Y}_1,\ldots,\mathcal{Y}_m$ of metric spaces such that $\mathcal{Y}_{i-1}$ is $R_{k+i}$-decomposable over $\mathcal{Y}_i$ for every $i\in\{1,\ldots,m\}$. Let $\mathcal{X}_{k+i} = \{X \cap Y \colon X \in \mathcal{X}_k,Y \in\mathcal{Y}_i\}$ for $i\in\{1,\ldots,m\}$. Then $\mathcal{X}_{k+i-1}$ is $R_{k+1}$-decomposable over $\mathcal{X}_{k+i}$ and $\mathcal{X}_{k+m}$ is uniformly bounded. Therefore $\mathcal{X}$ has straight finite decomposition complexity.
\end{proof}

\end{document}